\documentclass[amscd,amssymb,verbatim,10pt]{amsart}
\usepackage[fontsize = 10bp]{fontsize}
\usepackage{scalerel}
\usepackage{stackengine,wasysym}
\usepackage{bbm}

\usepackage{amssymb,latexsym, amsmath, amscd, array, graphicx, pb-diagram}
\usepackage{hyperref}
\usepackage{color}

\usepackage{tikz-cd}

\usepackage{color}
\hypersetup{
    colorlinks=true,
    linkcolor=red,
    urlcolor=black,
    citecolor=blue
           }

\usepackage{amsthm}
\usepackage{hyperref}
\usepackage[all]{xy}
\usepackage[scr]{rsfso}

\swapnumbers

\theoremstyle{plain}

\newtheorem{thm}{Theorem}[section]
\newtheorem{theorem}[thm]{Theorem}

\newtheorem{lemma}[thm]{Lemma}

\newtheorem{proposition}[thm]{Proposition}
\newtheorem{corollary}[thm]{Corollary}
\newtheorem{example}[thm]{Example}

\theoremstyle{definition}

\newtheorem{remark}[thm]{Remark}

\newtheorem{ex}[thm]{Example}

 \newcommand{\Wi}{\widetilde}

\DeclareMathOperator{\cat}{{\mathsf{cat}}}

\DeclareMathOperator{\rank}{{\mathsf{rank}}}
\DeclareMathOperator{\cu}{{\mathsf{cup}}}
\DeclareMathOperator{\dcat}{{\mathsf{dcat}}}

\DeclareMathOperator{\cd}{{\mathsf{cd}}}

\usepackage{enumitem}

\def\Im{\protect\operatorname{Im}}

\makeatletter
\def\@tocline#1#2#3#4#5#6#7{\relax
  \ifnum #1>\c@tocdepth % then omit
  \else
    \par \addpenalty\@secpenalty\addvspace{#2}%
    \begingroup \hyphenpenalty\@M
    \@ifempty{#4}{%
      \@tempdima\csname r@tocindent\number#1\endcsname\relax
    }{%
      \@tempdima#4\relax
    }%
    \parindent\z@ \leftskip#3\relax \advance\leftskip\@tempdima\relax
    \rightskip\@pnumwidth plus4em \parfillskip-\@pnumwidth
    #5\leavevmode\hskip-\@tempdima
      \ifcase #1
       \or\or \hskip 1em \or \hskip 2em \else \hskip 3em \fi%
      #6\nobreak\relax
    \hfill\hbox to\@pnumwidth{\@tocpagenum{#7}}\par% <---- \dotfill -> \hfill
    \nobreak
    \endgroup
  \fi}
\makeatother

\def\scr{\mathcal}

\def\B{{\scr B}}
\def\O{{\scr O}}
\def\C{{\mathbb C}}
\def\Z{{\mathbb Z}}
\def\Q{{\mathbb Q}}
\def\R{{\mathbb R}}

\def\rpn{{\R P^n}}

\def\1{\hbox{\rm\rlap {1}\hskip.03in{\rom I}}}
\def\Bbbone{{\rm1\mathchoice{\kern-0.25em}{\kern-0.25em}
{\kern-0.2em}{\kern-0.2em}I}}

\def\wt{\widetilde}

\def\cZ{\mathcal{Z}}

\usepackage{comment}

\long\def\forget#1\forgotten{} %

\newcommand\ver[1]{\marginpar{\tiny Changed in Ver \VER}}

\date{\today}

\begin{document}

\begin{abstract}
We show that in the presence of a geometric condition such as non-negative Ricci curvature, the distributional category of a manifold may be used to bound invariants, such as the first Betti number and macroscopic dimension, from above. Moreover, \`a la Bochner, when the bound is an equality, special constraints are imposed on the manifold. We show that the distributional category of a space also bounds the rank of the Gottlieb group, with equality imposing constraints on the fundamental group. These bounds are refined in the setting of cohomologically symplectic manifolds, enabling us to get specific computations for the distributional category and LS-category.
\end{abstract}

\title[Bochner-type Theorems for Distributional Category]{Bochner-type Theorems for Distributional Category}

\author[E.~Jauhari]{Ekansh~Jauhari}

\author[J.~Oprea]{John~Oprea}

\address{Ekansh Jauhari, Department of Mathematics, University of Florida, 358 Little Hall, Gainesville, FL 32611, USA.}

\email{ekanshjauhari@ufl.edu}

\address{John Oprea, Department of Mathematics, Cleveland State University, Cleveland, OH 44115, USA.}

\email{jfoprea@gmail.com}



\maketitle

\section{Introduction}
While Lusternik--Schnirelmann (LS) category was originally invented to provide an estimate for the number of
critical points of a smooth function on a manifold, it soon became apparent that the concept connected
in non-trivial ways with other topological invariants (see~\cite{CLOT}). The more recent development of the idea of
topological complexity and its relation to the motion planning problem (see~\cite{Far}) spawned a re-evaluation of
LS-category in that context, and this resulted in a new ``categorical'' notion called \emph{distributional category}
(see~\cite{DJ} and~\cite{KW}, where the term \emph{analog category} is used in the latter). LS-category and distributional category are
most certainly different. For instance, the LS-category of any real projective space $\rpn$ is $n$ while the distributional
category is $1$ for all $n$. Yet, intriguingly, sometimes $\dcat(X)=\cat(X)$ (where we have used the obvious notations).
In particular, from~\cite{KW} (also see~\cite{Dr2}), we know that for a torsion-free discrete group $\pi$ with finite cohomological
(and homotopy) dimension $k$, we have $\dcat(\pi)=\cat(\pi)=k$. In~\cite{Ja}, the first author gave some conditions under which $\dcat(M)=\cat(M)$ holds when $M$ is a closed (not necessarily aspherical) manifold. In this paper, we would
like to explore other situations where the above equality holds, or at least where natural topological or geometric conditions imply
something interesting about $\dcat$. We refer to these conditions as \emph{Bochner-type conditions} for reasons
that will become apparent below.

The results of our paper suggest close relationships between distributional category and various classical invariants of interest, 
such as the first Betti number, the Gottlieb group, LS-category, macroscopic dimension, etc. Furthermore, $\dcat$ affords us an opportunity to improve 
upon earlier results in the literature that connect these classical invariants (see, for example,~\cite{Op1},~\cite{Op2},~\cite{OS}, and~\cite{DDJ}). Our results are strongest in the setting of $c$-symplectic manifolds and non-negatively curved manifolds, where they also give us 
various new estimates and computations of $\dcat$. In particular, unlike the results of~\cite{DJ} and~\cite{Ja}, we obtain first 
examples of \emph{inessential manifolds} $M$ (in the sense of~\cite{Gr}) for which $\dcat(M)=\cat(M)$ holds due to various 
topological, group-theoretical, and geometric conditions rather than purely cohomological ones.

Just as LS-category has found applications in many areas of topology, geometry, and
critical point theory, we hope that the results of our paper will hint at similar interconnectedness of distributional category, and at least give some intuition about its behavior
%the behavior of distributional category 
on spaces with extra structures. We also mention several directions for further study that the reader may find interesting (see, for instance,  Remark~\ref{rem:pi1Bieberbach}).

%\subsection*{Notations and conventions} 
In this paper, all topological spaces are assumed to be path-connected. 
%We use the symbol ``$=$" to denote homeomorphisms of spaces and isomorphisms of groups, and the symbol ``$\simeq$" to denote homotopy equivalences. 

\section{Preliminaries}
In order to make this paper somewhat self-contained, we must review some foundational topics.
\subsection{Topological basics} 
%Recall that the \textit{$k$-th Betti number} of $X$, denoted $b_k(X)$, is defined to be the rank of $H_k(X;\Z)$. Equivalently, $b_k(X)=\dim_{\Q}(H_k(X;\Q))$. 
Let $R$ be a commutative ring with unity. 
Then the \emph{cup-length} of the cohomology ring $H^*(X;R)$, denoted $c\ell_R(X)$, is defined to be the largest integer $k$ 
for which there exist $\alpha_i\in H^{s_i}(X;R)$, $1\le i\le k$, $s_i>0$,
such that $\alpha_1\alpha_2\cdots\alpha_k\ne 0$. In this paper, we reserve the notation $\cu(X)$ for the rational cup-length $c\ell_{\Q}(X)$. Let $\Gamma$ be a discrete group. Then its \textit{cohomological dimension}, denoted $\cd(\Gamma)$, is defined to be the largest integer $k$
such that $H^k(\Gamma,A)\ne 0$ for some $\Z\Gamma$-module $A$. 
%For example, $\cd(\Z^n)=n$ and if $\Gamma$ has torsion, then $\cd(\Gamma)$ is infinite,~\cite{Br}. 
The Eilenberg--Ganea theorem and the Stallings--Swan theorem (see~\cite{Br}) together say that for $\cd(\Gamma)=k \neq 2$, there is an
Eilenberg--Mac~Lane space $K(\Gamma,1)$ of dimension $k$, and this is the smallest such dimension.

\subsection{LS-category}
The \emph{Lusternik--Schnirelmann category} (LS-category) of a space $X$, denoted $\cat(X)$, is defined to be the smallest integer
$n$ for which there is a covering $\{U_i\}$ of $X$ by $n+1$ open sets that are contractible in $X$. In particular,
$\cat(X)=0$ if and only if $X$ is contractible. The following are some well-known facts about Lusternik--Schnirelmann
category, whose proofs can be found in~\cite{CLOT}.
\begin{theorem}\label{clotth}
Let $X$ be CW a complex.
\vspace{-1.5mm}
\begin{enumerate}
\itemsep -0.4 em
    \item If $Y$ is homotopy equivalent to $X$ (i.e., $Y\simeq X$), then $\cat(Y)=\cat(X).$
        \item If $X$ is $r$-connected for $r\ge 0$, then $\cat(X)\le\tfrac{\dim(X)}{r+1}$.
        \item For any CW complex $Y$, $\cat(X\times Y)\le \cat(X)+\cat(Y)$.
        \item If $X'$ is a covering space of $X$, then $\cat(X')\le \cat(X)$.
        \item For any ring $R$, $c\ell_R(X)\le\cat(X)$. In particular, $\cu(X)\le\cat(X)$.
\end{enumerate}
\end{theorem}

These properties of $\cat$ will be quite useful for our purpose. In particular, Properties (2) and (5) tell us immediately that
for each $n$, $\cat(\R P^n)=n=\cat(\C P^n)$. A harder fact is that $\cat(X)=1$ if and only if $X$ is a co-$H$-space, see~\cite{CLOT}.

\subsection{Distributional category}
Recently, a probabilistic variant of LS-category was introduced in~\cite{DJ} and~\cite{KW}. This is the \emph{distributional category}
(or \emph{analog category} in~\cite{KW}). We refer to ~\cite{KW} and~\cite{Ja} for motivations behind this new invariant.
Given a pointed metric space $(X,x_0)$, the \emph{distributional category} of $X$, denoted $\dcat(X)$, is defined to be the smallest integer
$n$ for which there exists a map $H:X\to \B_{n+1}(P_0(X))$ such that $H(x)(0)\in \B_{n+1}(P(x,x_0))$ for each $x\in X$.
Here, $P_0(X)=\{\phi\in X^{[0,1]}\mid \phi(1)=x_0\}$ and $P(x,x_0)=\{\phi\in P_0(X)\mid \phi(0)=x\}$ have the (metrizable) subspace
topology coming from the compact-open topology of $X^{[0,1]}$, and $\B_{n+1}(Z)$ is the space of probability measures on
$Z$ supported by at most $n+1$ points, equipped with the Lévy--Prokhorov metric,~\cite{Pr}. We note that like $\cat(X)$ (see~\cite[Lemma 1.25]{CLOT}), $\dcat(X)$ is also independent of the choice of the basepoint $x_0\in X$, see also~\cite[Definition 3.1]{DJ}.

For a proof of the following, we refer to~\cite{DJ}.
\begin{theorem}[\protect{Dranishnikov--Jauhari}]\label{djth} Let $X$ be a CW complex.
    \vspace{-1.5mm}
\begin{enumerate}
\itemsep -0.4em
    \item $\dcat(X)\le\cat(X)$.
    \item If $Y\simeq X$, then $\dcat(Y)=\dcat(X)$.
    \item If $X'$ is a covering space of $X$, then $\dcat(X')\le \dcat(X)$.
    \item If $X$ is a finite CW complex, then $\cu(X)\le\dcat(X)$.
\end{enumerate}
\end{theorem}

\begin{ex}
%Just like LS-category, $\dcat(X)=0$ if and only if $X$ is contractible, and $\dcat(S^n)=1$. 
For each $n$, $\dcat(T^n)=\cat(T^n)=n$. 
If $X$ is a closed surface, then $\dcat(X)=1$ if $X\simeq S^2$ or $X\simeq \R P^2$, and $\dcat(X)=\cat(X)=2$ otherwise.
\end{ex}

In general, the distributional category and the LS-category of a space may differ. Indeed, $\dcat(\R P^n)=1<n=\cat(\R P^n)$ for
each $n\ge 2$, and for each finite group $\Gamma$, we have that $\dcat(\Gamma)<|\Gamma|<\infty=\cat(\Gamma)$, see~\cite{KW} (and also~\cite{DJ}).
The following is the main result of~\cite{KW} (see also~\cite{Dr2}).

\begin{theorem}[\protect{Knudsen--Weinberger}]\label{kwth}
If $\Gamma$ is a torsion-free discrete group, then $\dcat(\Gamma)=\cat(\Gamma)=\cd(\Gamma)$.
\end{theorem}
In particular, $\dcat(M)=\cat(M)=\dim(M)$ for closed aspherical manifolds (hence, nilmanifolds). This result has an analog for non-aspherical spaces. A closed manifold $M$ with $\cat(M)=\dim(M)$ is said to be
\emph{essential}~\cite{Gr} (see also ~\cite{KR}). In~\cite{Ja}, it was shown that $\dcat(M)=\cat(M)=\dim(M)$ for an essential manifold $M$ if $\pi_1(M)$ is
torsion-free and $M$ obeys an extra homological condition called the \emph{cap property}. This result may be extended to
(generalized) connected sums under certain ``cap property'' conditions (see~\cite{Ja}).

\subsection{$c$-symplectic manifolds}\label{prelimcsympl}
A $2n$-dimensional manifold $M$ is called \textit{$c$-symplectic} if there exists a cohomology class $\omega\in H^2(M;\Q)$ such that
$\omega^n\ne 0$. The appellation $c$-symplectic is used because these spaces ``mimic" symplectic manifolds cohomologically,~\cite{LO}.
While each symplectic manifold is $c$-symplectic, the converse is not true. Indeed, $\C P^2\#\C P^2$ is $c$-symplectic but not symplectic.
Some other examples of $c$-symplectic manifolds include the Calabi--Yau manifolds~\cite{Yau} (which, of course, are K\"ahler)
and symmetric products of closed orientable surfaces~\cite{Mac} (see also~\cite{DDJ}).
If $M$ is a closed $2n$-symplectic manifold and $M$ is simply connected, then $\cu(M)=\dcat(M)=\cat(M)=n$ holds due to Theorems~\ref{clotth}
and~\ref{djth}. If $M$ is not simply connected, then both $\cat$ and $\dcat$ can attain any integer value between $n$ and $2n$. 
Indeed, for any $0\le k \le n$, one can start with a closed simply connected symplectic $(2n-2k$)-manifold $N$ and create $M=N\times T^{2k}$ with
$\dcat(M)=\cat(M)=n+k$. For some special $c$-symplectic manifolds $M$, namely the \emph{symplectically aspherical manifolds},
it was shown in~\cite{RO} that $\cat(M)=\dim(M)$. For a symplectically aspherical manifold $M$ satisfying the ``cap property",
$\dcat(M)=\cat(M)=\dim(M)$ was shown in~\cite{Ja}.

\subsection{Gottlieb groups and Hurewicz rank}\label{gott}
The \textit{Gottlieb group} of a space $X$, denoted $G(X)$, is defined to be the subgroup of $\pi_1(X)$ consisting of the
elements $\alpha\in \pi_1(X)$ that have associated maps $A:S^1\times X\to X$ with $A|_{S^1} \simeq \alpha$ and $A|_X \simeq 1_X$.
This is equivalent to $G(X)$ being the image of the evaluation map $\text{ev}_*\colon \pi_1(X^X,1_X) \to \pi_1(X)$, where $\text{ev}(f)=f(x_0)$ for some $x_0\in X$. Since $\text{ev}_*$ is the connecting map for the universal fibration $X \to \text{Baut}_*(X) \to \text{Baut}(X)$ (see~\cite{Got1}), we see that 
$G(X) \subset \cZ\pi_1(X)$, where $\cZ\pi_1(X)$ is the center of $\pi_1(X)$. In the above setting for $G(X)$, if we replace $S^1$ by $S^n$ and $\pi_1$ 
by $\pi_n$, then we get the $n$-th Gottlieb group of $X$, denoted $G_n(X)$. The following are some basic properties of $G(X)$, see~\cite{Got1}.
\vspace{-1.5mm}
\begin{enumerate}
\itemsep -0.4em
\item $G(X)$ is equal to the center of $\pi_1(X)$ if $X$ is an aspherical space, and it is equal to the full fundamental group if $X$ is an $H$-space.
\item Let $p:X'\to X$ be a covering of $X$ and $p_*:\pi_1(X')\to \pi_1(X)$ be the induced homomorphism. If $\alpha\in \pi_1(X')$ and
$p_*(\alpha)\in G(X)$, then $\alpha\in G(X')$.
\item Let $K$ be a compact Lie group acting on $(X,x_0)$ and let $\mathcal{O}:K\to X$ be the orbit map defined as $g\mapsto g\cdot x_0$.
Then $\text{Im}(\mathcal{O}_*)\subset G(X)$.
\item For any $X$ and $Y$, $G(X\times Y)=G(X)\oplus G(Y)$.
\item If $X \to E \to B$ is a fibration, then there is a map to the universal fibration (as above) which induces the commutative diagram
\[
\xymatrix{
\pi_2(B) \ar[rr] \ar[dr]_-{\partial_*} 
& 
&
\pi_2(\mathrm{Baut}(X)) \cong \pi_1(X^X,1_X) \ar[dl]^-{\rm{ev}_*} 
\\
&
\pi_1(X)
&
&}
\]
that shows that $\Im(\partial_*) \subset \Im(\mathrm{ev}_*) \subset G(X)$. 
\end{enumerate}

Here is an example of the influence of the Gottlieb group on the structure of a space. Suppose $X$ is a topological space such that $H_1(X;\Z)$ is finitely generated. Let $h:\pi_1(X)\to H_1(X;\Z)$ be the Hurewicz homomorphism.
Then the \textit{Hurewicz rank} of $X$ is defined to be the number of $\Z$-summands of $H_1(X;\Z)$ which are contained in $h(G(X))$.
We refer to~\cite{Got2},~\cite{Op2} for a proof of the following.
\begin{theorem}[\protect{Gottlieb}]\label{hurewicz}
Let $X$ be a space such that $H_*(X;\Z)$ is finitely generated. If the Hurewicz rank of $X$ is $s$, then $X\simeq T^s\times Y$ for some space $Y$.
\end{theorem}

\section{Related results}\label{related results}
In this section, we recall some results from the literature that we aim to improve and/or extend in this paper. The most basic result for us
is the classical result of S.~Bochner that a closed $n$-manifold $M$ with non-negative Ricci curvature has $b_1(M) \leq \dim(M)$
and, perhaps most importantly, if $b_1(M)=\dim(M)$, then $M$ must be (diffeomorphic to) a flat torus. The proof given by Bochner was
analytic in nature, and a modern proof involving the Dirac operator may be found in~\cite{LM}. In fact, a simpler proof follows from the
powerful splitting theorem of Cheeger and Gromoll~\cite{CG}, which is stated as follows.
\begin{theorem}[\protect{Cheeger--Gromoll}]\label{cgth}
If $M$ is a closed $n$-manifold with non-negative Ricci curvature, then there exists a finite covering space $M'$ of $M$ and a
diffeomorphism $M'\cong T^r\times W$, where $W$ is a closed simply connected manifold.
\end{theorem}
We will refer to such a finite cover $M'$ as a \textit{Cheeger--Gromoll splitting} of $M$.

\begin{remark}
If $\pi_1(M)$ is finite, then $M'$ is the universal cover of $M$. So, Theorem~\ref{cgth} is not interesting in this case. Hence, we will
consider only those closed manifolds with non-negative Ricci curvature that have infinite fundamental groups.
\end{remark}

Using the Cheeger--Gromoll splitting, Bochner's theorem was improved by the second author in~\cite{Op1} as follows.

\begin{theorem}[\protect{Oprea}]\label{opth1}
Let $M$ be a closed $n$-manifold with non-negative Ricci curvature such that $\pi_1(M)$ is infinite. If $M'\cong T^r\times W$ is a
Cheeger--Gromoll splitting of $M$, then
    %\[
    $b_1(M)\le \cat(M)-\cu(W)$.
    %\]
Moreover, $b_1(M) = \cat(M)$ if and only if $M$ is a torus. Furthermore, if $M$ is also a closed $c$-symplectic $2n$-manifold, 
then  the Cheeger--Gromoll splitting takes the form $M'\cong T^{2k}\times W$, where $W$ is a closed simply connected $c$-symplectic 
manifold, and
 %\[
    $b_1(M)\le 2\cat(M)-\dim(M)$.
    %\]
\end{theorem}

\begin{remark}\label{transfer}
The proof of Theorem~\ref{opth1} uses the fact that the homomorphism $p^*:H^*(M;\Q)\to H^*(M';\Q)$ induced by the finite covering map $p:M'\to M$ is a split monomorphism due to the corresponding transfer map in cohomology, see~\cite[Sections III.9 \& III.10]{Br}. As a result, $b_1(M)\le b_1(M')$.
\end{remark}

In the case when $M$ is a closed $c$-symplectic manifold, Theorem~\ref{opth1} uses the following lemma from~\cite{Op1}, which we shall use below.

\begin{lemma}\label{imp1}
If $p:M'\to M$ is a finite covering map and $M$ is a $c$-symplectic $2n$-manifold with a cohomology class $\omega \in H^2(M;\Q)$,
then $M'$ is also $c$-symplectic with $\omega\in H^{2}(M';\Q)$. Moreover, if $M'\cong T^r\times W$ for some closed simply 
connected manifold $W$, then $r=2k$, and $w=w_1\oplus w_2$ for cohomology classes $w_1\in H^2(T^{2k};\Q)$ and $w_2\in H^2(W;\Q)$ satisfying $w_1^k\ne 0$ and $w_2^{n-k}\ne 0$.
%\vspace{-1.5mm}
%\begin{enumerate}
%\itemsep -0.38 em
%\item $r=2k$, and
%\item $w=w_1\oplus w_2$ for cohomology classes $w_1\in H^2(T^{2k};\Q)$ and $w_2\in H^2(W;\Q)$ satisfying $w_1^k\ne 0$ and $w_2^{n-k}\ne 0$.
%\end{enumerate}
\end{lemma}

\begin{remark}\label{rem:csympexistence}
The simplest examples of closed $c$-symplectic (actually symplectic) manifolds having non-negative Ricci (actually sectional) curvature are the products $M^{2n}=T^{2k} \times N^{2n-2k}$ for closed simply connected symplectic $(2n-2k)$-manifolds $N$. Here, we have $b_1(M)=2k < n+k = \cat(M)$.
%$M=T^2\times S^2$, and here we have $b_1(M)=2 < 3 = \cat(M)$. Of course, this example generalizes to any product $T^{2k} \times N^{2n-2k}$,
%where $N$ is any simply connected symplectic $(2n-2k)$-manifold. 
For such a product manifold $M$, the 
second estimate of Theorem~\ref{opth1} is sharp, but $M$ is \emph{not} a torus. Besides these examples, there are the Calabi--Yau manifolds with infinite fundamental groups~\cite{Yau} (although most physicists take their fundamental groups to be finite).
\end{remark}

The paper~\cite{FH}, which formulated a minimal model approach to rational LS-category, essentially re-invigorated the subject.
One of the main results of~\cite{FH} relates the ``size" of (all) the Gottlieb groups of a space with the LS-category of its rationalization 
(in the simply connected case) as follows: if $X$ is a simply connected topological space, then $\dim(G_*(X)\otimes_{\Z}\Q)\le\cat(X_{\Q})$,
where $X_{\Q}$ is the rationalization of $X$. For non-simply connected spaces, an integral analog of the above result for the
first Gottlieb group was obtained in~\cite{Op2} using Theorem~\ref{hurewicz}.

\begin{theorem}[\protect{Oprea}]\label{opth3}
If $A$ is a finitely generated free abelian subgroup of $G(X)$, then $\rank(A)\le \cat(X)$. 
In particular, if $G(X)$ is finitely generated free abelian, then $\rank(G(X))\le\cat(X)$.
Moreover, if $X$ is a closed $n$-manifold, then equality holds if and only if $X\simeq T^n$.
\end{theorem}

\section{Extensions for distributional category}\label{genres}
In this section, we will see that $\dcat$ allows us to improve the results of Section~\ref{related results}.
As motivation, recall that a compact Riemannian manifold $M^n$ is \emph{flat} if the sectional curvature at each
point is zero. It is known that $M^n$ is flat if and only if $M \cong \R^n/G$, where $G$ is a \emph{Bieberbach group} (see below) acting freely on $\R^n$ and $\cong$ denotes ``is isometric to''. Hence, $M$ is 
aspherical with fundamental group $\pi_1(M)\cong G$. Bieberbach's Theorem characterizes a Bieberbach group 
$G$ as one having a maximal normal free abelian subgroup $N \cong \Z^k$ such that $G/N = Q$ is finite.
In~\cite{Va}, Vasquez removed the maximal condition by showing that any free abelian $N$ with $G/N$
finite has a normal maximal free abelian centralizer of the same rank. Furthermore, using the fact that a finite index subgroup contains a normal finite index subgroup (namely the normal core), it can be shown that $G$ is the fundamental group  of a compact flat Riemannian $n$-manifold if and only if
\hspace{-1.5mm}
\begin{enumerate}
\itemsep -0.4em
\item $G$ is torsion-free,
\item $G$ contains a subgroup $N$ which is free abelian of rank $n$, and
\item the index $[G:N]$ is finite.
\end{enumerate}
Hence, if $M$ is flat, then $\pi_1(M)=G$ obeys these conditions. The converse was shown in 
\cite{FaHs}; that is, a closed (connected) manifold $M^n$ ($n \not = 3,4$)\footnote{For the case $n=3$, see~\cite{Wolf}, and for the case $n=4$, see~\cite{Lambert}.} supports a flat
Riemannian structure if and only if $M^n$ is aspherical (with $\cd(\pi_1(M))=n$) and $\pi_1(M)$ contains an abelian 
subgroup of finite index. From Section~\ref{gott}, we see that a flat manifold $M$ has the Gottlieb group given by the center of the fundamental group, i.e., 
$G(M) = \cZ \pi_1(M)$. Moreover, by Theorems~\ref{kwth} and~\ref{opth3}, we have 
$$\rank(\cZ\pi_1(M)) = \rank(G(M)) \leq \cat(M) = \dcat(M).$$
In~\cite{HiSa}, it was shown that $\rank(\cZ\pi_1(M)) = b_1(\pi_1(M))=b_1(M)$, so we see that we also have
%\[
$b_1(M) \leq \dcat(M)$,
%\]
an analog to Theorem~\ref{opth1}. In this case, the condition 
%\[
$b_1(M)=\cat(M) = \dcat(M) = \dim(M) = n$
%\]
then says that $\pi_1(M)$ contains a \emph{central} free abelian subgroup $\Z^n$. This can only happen if $\pi_1(M)$
is itself free abelian of rank $n$ (essentially by the same argument as in the proof of Theorem~\ref{new2} (3) given below). Hence, $M \cong T^n$. 

Since flat manifolds have non-negative Ricci curvature, they provide non-trivial examples of such manifolds with non-trivial Gottlieb groups, where we can also determine $\dcat$. Also, in~\cite{DHS} 
many flat manifolds were described that are K\"ahler (hence symplectic), so we have such non-trivial examples of $c$-symplectic nature too. Finally, we note that many of these flat K\"ahler manifolds have non-trivial centers (determined by their non-zero first Betti numbers), so they have non-trivial Gottlieb groups as well. 

Now, let us go beyond flat manifolds to consider manifolds with non-negative Ricci curvature. Our first result is a direct analog of Theorem~\ref{opth1} above for $\dcat$.
\begin{theorem}\label{new1}
Let $M$ be a closed $n$-manifold with non-negative Ricci curvature such that $\pi_1(M)$ is infinite.
If $M'\cong T^r\times W$ is a Cheeger--Gromoll splitting of $M$, then
\[
b_1(M)\le \dcat(M)-\cu(W).
\]
Moreover, $b_1(M) = \dcat(M)$ if and only if $M$ is a torus.
\end{theorem}

\begin{proof}
In view of Remark~\ref{transfer}, we have that $b_1(M)\le b_1(M')$. Since $W$ is simply connected, $b_1(M') = b_1(T^r)=r=\cu(T^r)$.
Now, Theorem~\ref{djth} gives us the following series of inequalities:
\[
\dcat(M)\ge \dcat(M')\ge \cu(M')=r+\cu(W)\ge b_1(M)+\cu(W).
\]
If $M$ is a torus, then $b_1(M)=\cat(M)=\dcat(M)$ by Theorems~\ref{opth1} and~\ref{kwth}. Conversely,
if $b_1(M)=\dcat(M)$, then the above inequalities give $\cu(W)=0$. Since $W$ is closed and simply connected (hence orientable),
this implies $W\simeq \ast$. Hence, $M'\cong T^n$. This means $M$ is a closed aspherical $n$-manifold. Therefore,
$b_1(M)=\dcat(M)=\cat(M)=n$ by Theorem~\ref{kwth}. But this is the situation of Bochner's theorem, so
we obtain $M\cong T^n$.
\end{proof}

We now obtain an analog of Theorem~\ref{opth3} for $\dcat$. Notice, however, that we do not get the full analog of Theorem~\ref{opth3}
because we only have a lower bound for $\dcat$ given by the rational cup-length instead of the integral cup-length.

Recall that in geometric
group theory (see~\cite{BB}), a discrete group $\Gamma$ is said to be in $FH(R)$ (written $\Gamma \in FH(R)$) for a commutative ring $R$ if
$\Gamma$ acts freely, properly discontinuously, cellularly, and cocompactly on an $R$-acyclic space (i.e., a CW complex $Y$ with $H^j(Y;R)=0$
for $j > 0$). This is one step in the ladder of finiteness conditions for a group (also see~\cite{Br}).

\begin{theorem}\label{new2}
If $A\subset G(X)$ is a finitely generated free abelian group, then
\[
\rank(A)\le \dcat(X).
\]
In particular, if $G(X)$ is finitely generated free abelian, then $\rank(G(X))\le\dcat(X)$.
Moreover,
\vspace{-1.5mm}
\begin{enumerate}
\itemsep -0.4em
\item if $X$ is compact and $\rank(G(X)) = \dcat(X)$, then $\pi_1(X) \in FH(\Q)$ and, hence, $\Q$ has a finite
resolution by finitely generated \emph{free} $\Q\pi_1(X)$-modules;
\item if $X$ is a nilpotent space and $\rank(G(X)) = \dcat(X)$, then $X \simeq_\Q K(\pi_1(X),1)$;
\item if $X$ is a closed $n$-manifold and $G(X)$ has finite index in $\pi_1(X)$, then $\rank(G(X)) = \dcat(X)$ implies $X \simeq T^n$.
\end{enumerate}
\end{theorem}

\begin{proof}
Let $p:X'\to X$ be a regular covering map corresponding to $A\subset \pi_1(X)$, so that $\pi_1(X')\cong A$. In fact, by Property (2) in
Section~\ref{gott}, we have $\pi_1(X')=G(X')$ since $p_*(A)\subset G(X)$. Because $A$ is free abelian, we have $A\cong\Z^k$ for some $k$
so that the Hurewicz rank of $A$ is $k$. Thus, by Theorem~\ref{hurewicz}, we have $X'\simeq T^k\times Y$, where $Y$ is simply connected. Clearly,
Theorem~\ref{djth} gives us the following series of inequalities (where recall that $\cu$ denotes the rational cup-length):
\[
\dcat(X)\ge \dcat(X')\ge \cu(X')=k+\cu(Y)\ge k=\rank(A).
\]
Now, suppose $\rank(G(X)) = \dcat(X)$. Then we have $\cu(Y)=0$ so that $Y$ is rationally acyclic. But since $X' \simeq T^k \times Y$, 
we have for the universal cover $\wt{X}$ of $X$ that $\widetilde X \simeq \R^k \times Y$. Hence, $\wt{X}$ is also $\Q$-acyclic. 
Since the fundamental group of a space always acts properly discontinuously on the universal cover and $X$ is compact, we conclude 
that $\pi_1(X) \in FH(\Q)$.

%\marginpar{\textit{$\wt{X}$ is $\Q$-acyclic in (2) by proof of (1), right? If so, statement of (2) must say ``compact nilpotent".}}
Note that in the previous argument, we only needed $X$ to be compact to fit the definition of $FH(\Q)$. The fact that $Y$, and
hence $\wt{X}$, is $\Q$-acyclic does not depend on the compactness of $X$. Now, 
if $X$ is a nilpotent space (i.e., $\pi_1(X)$ is nilpotent and acts nilpotently on the homology of the universal cover), we can rationalize
it to get $X_\Q$ with the property that $\pi_1(X_\Q)$ is the Malcev $\Q$-completion $\pi_1(X)_\Q$ of the nilpotent group $\pi_1(X)$, and $\pi_j(X_\Q)\cong \pi_j(X) \otimes \Q \cong \pi_j(\widetilde X) \otimes \Q \cong \{1\}$ for each $j\ge 2$ because $\widetilde X$ is simply connected and 
$\Q$-acyclic. Therefore, $X_\Q\cong K(\pi_1(X_\Q),1) \cong K(\pi_1(X)_{\Q},1)$ is an Eilenberg--Mac~Lane space and $X$ itself is rationally homotopy 
equivalent to $K(\pi_1(X),1)$. This is often written as $X \simeq_\Q K(\pi_1(X),1)$. 
%\marginpar{\textit{Do you mean $X_\Q=K(\pi_1(X_{\Q}),1)$ in (2)? If not, then what is $\pi_1(X)_\Q$?}}

Now, suppose that $X$ is a closed $n$-manifold and $G(X)$ has finite index in $\pi_1(X)$. Then we see that $A\cong \Z^k$ 
(as above) also has a finite index in $\pi_1(X)$. Therefore, $X' \simeq T^k \times Y \to X$ is a finite covering and $X'$ is then 
a closed $n$-manifold as well. Thus, $Y$ is also a closed manifold of dimension $n-k$. Just as before, since $\rank(G(X)) = k = 
\dcat(X)$, we have $\cu(Y)=0$. Hence, $Y$ is $\Q$-acyclic and simply connected. This means that $X'$ is an orientable closed 
$n$-manifold. So, we must have $H_n(X';\Q)\cong \Q$. But since $Y$ is $\Q$-acyclic, this can only happen if $k=n$, $Y \simeq *$, 
and $X' \simeq T^n$. This then means that $X \simeq K(\pi_1(X),1)$, so that $\pi_1(X)$ is torsion-free. But then we
have a short exact sequence
\[
\Z^n \to \pi_1(X) \to F,
\]
where $\Z^n$ is central in $\pi_1(X)$ and $F$ is finite. Consider the associated $5$-term exact sequence (see
\cite[Corollary VII.6.4]{Br})
\[
H_2(\pi_1(X)) \to H_2(F) \to \Z^n/[\Z^n,\pi_1(X)] \to H_1(\pi_1(X)) \to H_1(F).
\]
Because $\Z^n$ is central, the middle term reduces to $\Z^n$ itself. The quotient $F$ is finite, so $H_2(F)$ is finite as well and 
therefore, the image of $H_2(F) \to \Z^n$ is zero and $\Z^n \to H_1(\pi_1(X))$ is an injection. This translates to an injection $\Z^n \to
\pi_1(X)/[\pi_1(X),\pi_1(X)]$ so that $\Z^n \cap [\pi_1(X),\pi_1(X)] = \{1\}$. But then we have an injection
$[\pi_1(X),\pi_1(X)] \to \pi_1(X)/\Z^n$, where $\pi_1(X)/\Z^n$ is finite by our hypothesis. Since as a subgroup of $\pi_1(X)$,
$[\pi_1(X),\pi_1(X)]$ is torsion-free, we must have $[\pi_1(X),\pi_1(X)]=\{1\}$. Therefore, $\pi_1(X)$ is a finitely generated
torsion-free abelian group. This implies that $\pi_1(X) \cong \Z^n$ and $X \simeq T^n$.
\end{proof}

\begin{remark}\label{rem:noconverse}
If $M$ is a closed nilmanifold, then $\dcat(M)=\dim(M)$ by Theorem~\ref{kwth}. Note that $\rank(G(M))\le\rank(\mathcal{Z}\pi_1(M))<\dim(M)$, except when $M$ is a torus, see~\cite[Appendix~B]{OS2}. 
%implies that In~\cite[Section 3.2]{Ja}, it was shown that $\dcat(\mathcal{KT})=4$, where $\mathcal{KT}=\mathcal H_3 \times S^1$ is the Kodaira--Thurston symplectic non-K\"ahler nilmanifold (and $\mathcal H_3$ is the $3$-dimensional Heisenberg nilmanifold). 
%given by upper-triangular $3\times 3$-matrices in $GL(\R,3)$ with $1$'s on the diagonal modulo such matrices in $GL(\Z,3)$). Now, 
%$\mathcal{KT}=K(\pi,1)$ since it is a nilmanifold and 
%$G(\mathcal{KT}) = \Z^2$ since the center of $\pi_1(\mathcal H_3)$ is $\Z$.
Therefore, being a compact Eilenberg--Mac~Lane space
is not sufficient to obtain $\rank(G(X)) = \dcat(X)$. 
%Also, notice that $b_1(\mathcal{KT})=3 \not = \rank(G(\mathcal{KT}))$ since nilmanifolds are not flat, just almost flat. 
\end{remark}

\begin{corollary}\label{cor:Grealize}
Suppose $M = L \times Y$  for a closed manifold $Y$ such that $\pi_1(Y)$ is finite with a trivial center
and a closed manifold $L$ with an H-space structure (for example, a closed Lie Group).  Then
%\[
$\dcat(M) > \rank(\pi_1(L))$.
%\]
\end{corollary}

\begin{proof}
First, note that such a $Y$ always exists since any finitely presented group can be realized as the fundamental group of a closed $4$-manifold.
For such an $L$, we have $G(L) = \pi_1(L)$. Because $\cZ\pi_1(Y)=\{1\}$, $G(M)=G(L)=\pi_1(L)$ has finite index in $\pi_1(M)$. 
By Theorem~\ref{new2} (3), since $M$ is not a torus, we must have $\rank(\pi_1(L))= \rank(G(M)) < \dcat(M)$.
\end{proof}

The shortness of the proof of Corollary~\ref{cor:Grealize} belies the fact that there does not seem to be another way of seeing $\dcat(M) > \rank(\pi_1(L))$. Note that, in particular, if $L = T^k$ above, we get $\dcat(M) > k$. Moreover,  the closed Lie group
$L$ can be chosen such that $\pi_1(L)$ is any given finitely generated abelian group using tori and the fact that the projective
special unitary group has $\pi_1(PSU(n)) \cong \Z/n\Z$.

\begin{remark}
If a torus $T^k$ acts on a closed $n$-manifold $M$ such that the map $\O_*: \pi_1(T^k)\cong\Z^k\to \pi_1(M)$ induced by the 
orbit map $\O$ is a monomorphism, the action is called an \textit{injective toral action}. These were studied extensively in~\cite{CR}. By (3) of Section~\ref{gott}, any such action has  $\text{Im}(\mathcal{O}_*) \subset G(M)$. Thus, $G(M)$ contains a copy of $\Z^k$.
So, we have the situation of Theorem~\ref{new2} and there is an associated covering $M'\simeq T^k\times N$ with $N$ simply connected. It was shown in~\cite{Op2}
%that if $M$ is a closed $n$-manifold with an injective toral action of the torus $T^k$, then by the arguments as above, 
that $\cat(M)\ge k+\cu(N)$, and $\cat(M)=k$ holds if and only if $k=n$ and
$M\simeq T^n$. Any flat manifold $M$ has an injective toral action by $T^k$, where $k=b_1(M)$. For $\dcat$, however, we only 
have conclusions as in Theorem~\ref{new2} for injective total actions.
\end{remark}

\begin{example}\label{exam:Grealize}\rm{
In~\cite{OS}, it was observed that any finitely generated abelian group can arise as the Gottlieb group of a closed
manifold. The construction is very simple. If $A = \Z^k \oplus \bigoplus_{i=1}^n \Z/s_i\Z$, then let
\[
M = T^k \times \prod_{i=1}^n S^3/(\Z/s_i\Z)
\]
where $\Z/s_i \Z \subset S^1 \subset S^3$ gives the factor space $S^3/(\Z/s_i\Z)$. For this type of factor space, the Gottlieb group is the center of $\Z/s_i\Z$ by \cite{Lang},~\cite{OpG}, so we see that $G(S^3/(\Z/s_i\Z))\cong \Z/s_i\Z$ itself. By Property (4) in Section~\ref{gott}, we then have $G(M)\cong A$. But now we notice the following two things about $M$.

First, since each quotient $S^3/(\Z/s_i\Z)$ is a closed orientable $3$-manifold, we have $H^3(S^3/(\Z/s_i\Z);\Q)\cong\Q$ and 
$H^k(S^3/(\Z/s_i\Z);\Q)=\{1\}$ in degrees $1$ and $2$ (this follows from the transfer in Remark~\ref{transfer}). Thus, $\cu(M)=k + n$. 
Therefore, we have that 
\[
\rank(G(M)) = k < k + n = \cu(M) \leq \dcat(M).
\]
In particular, the difference between $\rank(G(M))$ and $\dcat(M)$ can be arbitrarily large. Also, here we see an explicit example
where $\rank(G(M))=\dcat(M)$ holds exactly when $n=0$ and $M = T^k$ as in (3) of Theorem~\ref{new2}.

The second observation is that $M$ has non-negative Ricci curvature if $T^k$ is taken to be flat and each $S^3/(\Z/s_i\Z)$ has the 
associated spherical metric. Then Theorem~\ref{new1} gives the estimate
\[
k=b_1(M) \leq \dcat(M) - \cu\left(\prod_{i=1}^n S^3\right) = \dcat(M)-n,
\]
so we see the inequality $k+n \leq \dcat(M)$ from this viewpoint as well.
}
\end{example}

\begin{example}\label{exam:homogenspaces}\rm{
Suppose $L$ is a compact connected Lie group and $H$ is a connected closed subgroup. Then there is a Barratt--Puppe sequence
\[
\Omega(BH)\simeq H \to \Omega(BL) \simeq L \to L/H \to BH \to BL,
\]
so that the connecting homomorphism of the fibration $L/H \to BH \to BL$ may be taken to be $\partial_*\colon \pi_1(L) \cong \pi_1(\Omega(BL)) \cong \pi_2(BL) \to \pi_1(L/H)$.
This implies that $\text{Im}(\partial_*) \subset G(L/H)$ --- see Property (5) in Section~\ref{gott}. Since $H$ is connected, $\partial_*$ is an epimorphism. Thus, $G(L/H) \cong \pi_1(L/H)$.

As an explicit example, take $L=T^k \times K$ for a compact connected Lie group $K$, and let $H=S^1 \times C$, where $S^1$ is a factor of 
$T^k$ and $C$ is the maximal torus of $K$. Then 
\[
L/H \cong T^{k-1} \times K/C
\]
and $K/C$ is a closed simply connected K\"ahler manifold. Therefore, we have that $\cu(L/H)=\cu(T^{k-1})+\cu(K/C)=k-1+
\tfrac{\dim(K)-\rank(K)}{2} \leq \dcat(L/H)$, while
\[
\dcat(L/H) \leq \cat(L/H) \leq \cat(T^{k-1}) + \cat(K/C) = k-1 + \frac{\dim(K)-\rank(K)}{2}.
\]
Hence, we have the inequality
\[
\rank(G(L/H)) = k-1 < k-1 + \frac{\dim(K)-\rank(K)}{2}=\dcat(L/H).
\]
Here, we again have a construction where the difference between $\rank(G(L/H))$ and $\dcat(L/H)$ is arbitrary.
}
\end{example}

In the following section, we will refer to the cover $X'\simeq T^k\times Y$ mentioned in the proof of Theorem~\ref{new2} as a
\emph{Hurewicz splitting} of $X$ corresponding to $A\subset G(X)$.

\section{Extensions for $c$-symplectic manifolds}
Our results from Section~\ref{genres} illustrate that the techniques used to lower bound LS-category (see~\cite{Op1},~\cite{Op2}) by $b_1(M)$ and $\rank(G(X))$ remain useful in lower bounding distributional category as well. So, it is natural to ask whether $\dcat=\cat$ in the settings of Theorems~\ref{new1} and~\ref{new2}. In this section, we address this interesting question by explaining how our results from Section~\ref{genres} can be refined in the setting of closed $c$-symplectic manifolds and $\dcat=\cat$ can be obtained.

\subsection{Manifolds with non-negative Ricci curvature}\label{c-symp non-neg Ricci}

\begin{proposition}\label{newnew1}
Let $M$ be a closed $c$-symplectic Riemannian $2n$-manifold with non-negative Ricci curvature such that $\pi_1(M)$ is infinite.
If $M'\cong T^{2k}\times W$ is a Cheeger--Gromoll splitting of $M$, then
\[
b_1(M)\le 2k\le 2\dcat(M)-\dim(M).
\]
Moreover, $\dcat(M')=\cat(M')=n+k$.
\end{proposition}

\begin{proof}
First, we note that such a splitting exists, see Theorems~\ref{cgth} and~\ref{opth1}. Since $W$ is a closed simply-connected 
$c$-symplectic $(2n-2k)$-dimensional manifold, we have $\cu(W)=\cat(W)=n-k$,
see Section~\ref{prelimcsympl}. Thus, we get
\[
2\dcat(M)\ge 2\dcat(M')\ge 2\cu(M')=4k+2(n-k)=2k+\dim(M)
\]
in view of Theorem~\ref{djth}.

To show the second part, since $n+k = \cu(M') \leq \dcat(M')\le\cat(M')$, it suffices to prove that $\cat(M')\le n+k$.
Of course, Theorem~\ref{clotth} gives us
\[
\cat(M')\le \cat(T^{2k})+\cat(W)=2k+(n-k)= n+k.
\]
\end{proof}

\begin{remark}
    We note that without the non-negative Ricci curvature hypothesis on $c$-symplectic manifolds $M$ having infinite fundamental group, even the (relatively) weaker inequality $b_1(M)\le\dcat(M)$ from Theorem~\ref{new1} fails.
%fundamental group in Theorem~\ref{new1} and Proposition~\ref{newnew1} cannot be dropped. 
To see this, consider the $n$-th symmetric product $SP^n(M_g)$ of the closed orientable surface $M_g$ of genus $g>n\ge 2$, which is $c$-symplectic but cannot support a Riemannian metric of non-negative Ricci curvature,~\cite[Section 5]{DDJ}. It follows from~\cite[Section 6]{DDJ} that 
    \[
b_1(SP^n(M_g))=2g>2n=\dcat(SP^n(M_g))=\dim(SP^n(M_g)).
    \]
\end{remark}
For a proof of the following classical result, we refer to~\cite[Theorem VIII.3.1]{Br}.

\begin{theorem}[\protect{Serre}]\label{serreth}
Let $\Gamma'$ be a finite index subgroup of $\Gamma$. If $\Gamma$ is torsion-free, then $\cd(\Gamma)=\cd(\Gamma')$.
\end{theorem}

In the case when $\pi_1(M)$ is torsion-free, we can say more than Proposition~\ref{newnew1} by using Serre's theorem. 
This is a prime example of the type of geometric condition leading to the equality $\dcat=\cat$.

\begin{theorem}\label{newnewnew1}
Let $M$ be a closed $c$-symplectic Riemannian $2n$-manifold with non-negative Ricci curvature such that $\pi_1(M)$ is torsion-free.
If $M'\cong T^{2k}\times W$ is a Cheeger--Gromoll splitting of $M$, then
\[
\dcat(M)=\cat(M)=n+k.
\]
\end{theorem}

\begin{proof}
Because of Proposition~\ref{newnew1} and Theorem~\ref{clotth} above, we have the inequality
%\begin{equation}\label{eq1} 
$n+k=\dcat(M')=\cat(M')\le \cat(M)$.
%\end{equation}
The main result of~\cite{Dr1} says that
%\begin{equation}\label{eq2}
\[
\cat(M)\le\frac{\dim(M)+\cd(\pi_1(M))}{2} = \frac{2n+\cd(\pi_1(M))}{2}.
\]
%\end{equation}
Since $\pi_1(M)$ is torsion-free and $M'$ is a finite cover of $M$ (see Theorem~\ref{cgth}),
Serre's Theorem~\ref{serreth} implies that $\cd(\pi_1(M))=\cd(\pi_1(M'))$. Since $W$ is simply connected,
$\pi_1(M')\cong\pi_1(T^{2k})\cong\Z^{2k}$. Hence, $\cd(\pi_1(M))=2k$. Therefore, we get from the above two inequalities that 
$\cat(M)=n+k$. By Proposition~\ref{newnew1} and Theorem~\ref{djth}, we have
\[
n+k = \dcat(M') \leq \dcat(M) \leq \cat(M)=n+k.
\]
This completes the proof.
\end{proof}

The following is an immediate consequence of the argument in Theorem~\ref{newnewnew1}.

\begin{corollary}
Let $M$ be a closed $c$-symplectic $2n$-manifold with non-negative Ricci curvature such that $\pi_1(M)$ is torsion-free.
Then $\cd(\pi_1(M))$ is even and is the rank of the torus in the Cheeger--Gromoll splitting of $M$ implied by Theorem~\ref{cgth}.
\end{corollary}

\begin{remark}\label{rem:pi1Bieberbach}
By our earlier discussion of flat manifolds in Section~\ref{genres}, we see that imposing the condition in Theorem~\ref{newnewnew1} that 
$\pi_1(M)$ be torsion-free is equivalent to saying $\pi_1(M)$ is a Bieberbach group. This means that the classifying 
map of the universal cover $M \to K(\pi_1(M),1)$ may be taken to be a map $M \to N$ where $N$ is flat. We ask the
question of whether such a map has any topological consequences!
\end{remark}

\subsection{The Gottlieb group}
We begin by noting that from the proof of Lemma~\ref{imp1}, if $M$ is a closed $c$-symplectic $2n$-manifold,
$A$ is a finitely generated free abelian \emph{finite index} subgroup of $\pi_1(M)$, and $M'\simeq T^r\times Y$ is a 
Hurewicz splitting of $M$ corresponding to $A$, then $r=2k$ and $M'$ is a closed $c$-symplectic $2n$-manifold.

\begin{proposition}\label{newnew2}
Let $M$ be a closed $c$-symplectic $2n$-manifold. If $A\subset G(M)$ is a finitely generated, free abelian finite index subgroup of $\pi_1(M)$
and $M'\simeq T^{2k}\times Y$ is a Hurewicz splitting of $M$ corresponding to $A$, then $\dcat(M')=\cat(M')=n+k$.
\end{proposition}

\begin{proof}
First, we note from the proof of Theorem~\ref{new2} that $Y$ is simply connected. Furthermore,
$M'$ and $Y$ are $c$-symplectic manifolds of dimensions $2n$ and $2n-2k$, respectively, by Lemma~\ref{imp1}. Thus, we have that
$n-k=\cu(Y)=\cat(Y)$. Finally, Theorems~\ref{clotth} and~\ref{djth} give us the following series of inequalities:
\[
n+k= 2k+\cu(Y)\le \cu(M')\le\dcat(M')\le\cat(M')\le 2k+\cat(Y)=n+k.
\]
\end{proof}

\begin{corollary}\label{cor:nocenter}
If $M^{2n} = T^{2k} \times Y$ for a closed symplectic manifold $Y$ having finite $\pi_1(Y)$ with a trivial center, then $\dcat(M)\geq n+k$.
%\[
%\dcat(M) \geq \dcat(M') = \frac{\dim(M)}{2} + k.
%\]
\end{corollary}

\begin{proof}
By~\cite{Gom}, such a manifold $Y$ exists. We have $G(M) \cong  \Z^k$ since $\pi_1(Y)$ has trivial center, and $G(M)$ has finite index in
$\pi_1(M)$ since $\pi_1(Y)$ is finite. By Proposition~\ref{newnew2}, we get $\dcat(M') = \cat(M')= n+ k$. The standard
inequality $\dcat(M) \geq \dcat(M')$ completes the proof.
\end{proof}

As in Section~\ref{c-symp non-neg Ricci}, we can say more using Serre's Theorem~\ref{serreth} and the main result
of~\cite{Dr1} in the case when the groups are torsion-free.

\begin{theorem}\label{newnewnew2}
Let $M$ be a closed $c$-symplectic $2n$-manifold such that $\pi_1(M)$ is torsion-free. If $A\subset G(M)$ is a finitely generated,
free abelian finite index subgroup of $\pi_1(M)$ and $M'\simeq T^{2k}\times Y$ is a Hurewicz splitting of $M$ corresponding to
$A$, then
\[
\dcat(M)=\cat(M)=n+k.
\]
\end{theorem}

\begin{proof}
Because of Proposition~\ref{newnew2} and Theorem~\ref{clotth} above, we have the inequality $n+k=\dcat(M')=\cat(M')\le\cat(M)$. In light of Theorem~\ref{djth}, it suffices to prove $\cat(M)\le n+k$. Proceeding along the lines of the proof of
Theorem~\ref{newnewnew1}, we get
\[
\cat(M)\le\frac{2n+\cd(\pi_1(M'))}{2}
\]
due to Serre's theorem and Dranishnikov's theorem. But $M'\simeq T^{2k}\times Y$, and $Y$ is simply connected.
So, $\cd(\pi_1(M'))=2k$. This completes the proof.
\end{proof}

%A direct consequence of the computation in Theorem~\ref{newnewnew2} is the following.

%\begin{corollary}
%Let $M$ be a closed $c$-symplectic $2n$-manifold such that $\pi_1(M)$ is torsion-free. Then for each finitely
%generated, free abelian finite index subgroup $A$ of $\pi_1(M)$, there is a unique torus $T^{2k}$ whose product with a closed simply connected $(2n-2k)$-manifold gives a Hurewicz splitting of $M$ corresponding to $A$.
%with non-negative Ricci curvature. Then there is a unique torus $T^{2k}$ whose product with a closed simply connected $(2n-2k)$-manifold gives a Cheeger--Gromoll splitting of $M$ implied by Theorem~\ref{cgth}.
%\end{corollary}

Note that in~\cite{BLO}, as part of a wider study of manifolds with almost non-negative sectional curvature, it was shown that a closed
manifold $M$ with non-negative Ricci curvature obeys 
\[
\cat(M) \geq b_1(M) + e_0(\widetilde M),
\]
where $e_0$ is a rational homotopy invariant best described with the machinery of minimal models. Suffice it to say that it is always the case that $\cu(X)
\leq e_0(X)$ for simply connected $X$, and for an $X$ that is \emph{formal} in the sense of rational homotopy theory, $\cu(X) = e_0(X)$. Therefore, our bound 
\[
\dcat(M) \geq b_1(M) + \cu(W)
\]
in Theorem~\ref{new1} is tantalizingly close to saying that $\cat(M)=\dcat(M)$ for all such $M$ (since the Cheeger--Gromoll splitting 
$T^k \times W$ implies $W \simeq \widetilde M$). In both these results, if $\pi_1(M)$ is torsion-free, then $b_1(M)$ may be replaced 
by $\cd(\pi_1(M))$. We leave it to the reader to ponder whether $\dcat(M)=\cat(M)$ for such manifolds.

Finally, note that there is a minimal model formulation of the rational category $\cat_0(X)$, see~\cite{CLOT}. We pose the problem
of finding such a formulation for $\dcat_0(X)=\dcat(X)_\Q$ for $X$ a simply connected space.

\section{Distributional category and macroscopic dimension}\label{sec:macro}
Before we end, we would like to relate the distributional category to another very geometric type of invariant in the Bochner fashion.
The notion of the macroscopic dimension of Riemannian manifolds was introduced by Gromov in~\cite{Gr2}, in part to study 
the problem of the (non-)existence of Riemannian metrics of positive scalar curvature (see also~\cite{Dr3},~\cite{DDJ}). For a Riemannian manifold $X$, 
its \emph{macroscopic dimension}, denoted $\dim_{mc}X$, is the smallest integer $n$ for which there exists a continuous proper map 
$f:X\to K$ to an $n$-dimensional simplicial complex $K$ such that for some $c>0$, $\text{diam}(f^{-1}(y))<c$ for all $y\in K$.

\begin{remark}
By definition, $\dim_{mc}X\le \dim(X)$, and equality holds if $X=\R^n$. If $X$ is a compact manifold, then $\dim_{mc}X=0$. 
So, for a compact manifold $M$, we will talk about the macroscopic dimension of its universal Riemannian cover $\wt{M}$.
\end{remark}

For closed manifolds $M$, a strict improvement to the above inequality was recently obtained in~\cite[Theorem 7.18]{DDJ} 
using LS-category as follows:
    \[
    \dim_{mc}\wt M\le\cat(M)\le\dim(M).
    \]
We further improve this inequality in the case of manifolds having non-negative Ricci curvature\footnote{\hspace{0.5mm}There is a different notion of macroscopic dimension due to Dranishnikov (see~\cite{Dr3}), which is defined using Lipschitz maps instead of continuous maps. For closed manifolds $M$ having non-negative Ricci curvature, these two notions coincide on $\wt M$--- see~\cite[Proposition 7.10]{DDJ}.} using distributional category.
%If $M$ is a closed manifold with non-negative Ricci curvature, then it was shown in~\cite[Proposition 7.10]{DDJ} that $\dim_{mc}\wt M=\dim_{MC}\wt M$, where $\dim_{mc}\wt M=\dim(M)$ if and only if $M$ is flat.

\begin{theorem}\label{dcatmacdim}
Let $M$ be a closed Riemannian $n$-manifold with non-negative Ricci curvature such that $\pi_1(M)$ is infinite. If $M'\cong T^{r}\times W$ is a 
Cheeger--Gromoll splitting of $M$, then
    \[
    \dim_{mc}\wt M\le\dcat(M')\le\dcat(M).
    \]
Moreover, $\dim_{mc}\wt M=\dcat(M)$ if and only if $M$ is flat.
\end{theorem}

\begin{proof}
The proof of~\cite[Proposition 7.10]{DDJ} implies that $\dim_{mc}\wt M=r$. But we have from the proof of Theorem~\ref{new1} 
that $r+\cu(W)\le\dcat(M')$. Therefore,
    \[
    \dim_{mc}\wt M=r\le r+\cu(W)\le\dcat(M')\le\dcat(M).
    \]
If $M$ is flat, then $M$ is aspherical and $\dim_{mc}\wt{M}=\dim(M)=\dcat(M)$ by~\cite[Corollary 7.7]{DDJ} and Theorem~\ref{kwth}, respectively. Conversely, if $\dim_{mc}\wt M=\dcat(M)$, then we get from the above series of inequalities that $\cu(W)=0$ for the closed simply connected manifold $W$. This means $W\simeq\ast$ and hence, $M'\cong T^n$ (as in the proof of Theorem~\ref{new1}). Therefore, $M$ is a closed aspherical $n$-manifold and $\pi_1(M)$ is torsion-free. Since $\pi_1(M')\cong\Z^n$ has finite index in $\pi_1(M)$, the generalization of Vasquez's characterization~\cite{Va} mentioned in Section~\ref{genres} implies that there exists a closed flat $n$-manifold $N=K(\pi_1(M),1)$. Since $M=K(\pi_1(M),1)$, we must have $M\simeq N$. In the cases $n\ne 3,4$, the results of Farrell--Hsiang~\cite{FaHs} imply that the corresponding \textit{surgery structure set} is zero and so $M\cong N$ (see Section~\ref{genres}). For $n=3,4$,~\cite[Theorems 0.7 \& 0.11]{KL} imply that the flat manifold $N$ is \textit{strongly Borel}, which means the homotopy equivalence $N\to M$ above is homotopic to a homeomorphism. Therefore, in any case, $M$ is flat.
\end{proof}

\begin{remark}
    The fact that $\dim_{mc}\wt M = \dim(M)$ implies the flatness of $M$ was proved in~\cite[Proposition 7.10]{DDJ} in a completely different way, so once we showed in the proof above that $M$ is aspherical, we could have simply applied that result along with Theorem~\ref{kwth}. Nevertheless, it seems worthwhile to see how a more classical approach leads to the flatness conclusion.
\end{remark}

\begin{corollary}\label{similar}
    Suppose $M=N \times Y$ has non-negative Ricci curvature for a closed manifold $Y$ such that $\pi_1(Y)$ is finite and a closed aspherical manifold $N$. Then
    \[
    \dcat(M) > \dim(N).
    \]
\end{corollary}

\begin{proof}
    Observe that $\dim_{mc}\wt M\le\dim_{mc}\wt N + \dim_{mc}\wt Y = \dim(N)$ because $N$ is aspherical and $\wt Y$ is compact. Then by the definition of macroscopic dimension, we must have $\dim_{mc}\wt M=\dim(N)$. But because $M$ is not flat, Theorem~\ref{dcatmacdim} implies $\dim(N)=\dim_{mc}\wt M<\dcat(M)$.
%\\
%Now, suppose $\dim_{mc}\wt M<\dim(N)=n$. Then there is a continuous proper uniformly cobounded map $f:\wt M\to K^{n-1}$. Then because $\wt M = \wt N\times\wt Y$, we have $\wt N \hookrightarrow \wt M\to K^{n-1}$, implying $\dim_{mc}\wt N \le n-1$, which is a contradiction.
\end{proof}

\begin{remark}
The hypothesis of non-negative Ricci curvature in Theorem~\ref{dcatmacdim} is necessary.
%even for the weaker implication that $\dim_{mc}\wt M=\dcat(M)\implies M$ is aspherical.
%In the above setting, in particular, $\dim_{mc}\wt M=\dcat(M)\implies M$ is aspherical. However, for this implication, we cannot drop the non-negative Ricci curvature hypothesis on $M$. 
To see this, consider the $g$-th symmetric product $SP^g(M_g)$ of the closed orientable surface $M_g$ of genus $g\ge 2$ and note its following properties from~\cite{DDJ}: 
$SP^g(M_g)$ cannot support a metric of non-negative Ricci curvature, $\dim_{mc}\wt{SP^g(M_g)}=2g=\dcat(SP^g(M_g))$, and $\pi_2(SP^g(M_g))$ is infinite.
\end{remark}

\begin{example}\rm{
We note that for the non-negatively curved manifolds of the type
\[
M = T^k \times \prod_{i=1}^n S^3/(\Z/s_i\Z)
\]
from Example~\ref{exam:Grealize}, we have $\dim_{mc}\Wi M\le k+\sum_{i=1}^n\dim_{mc}S^3=k< k+n\le\dcat(M)$ as in Corollary~\ref{similar}. Therefore, the difference between $\dim_{mc}\Wi M$ and $\dcat(M)$ can be arbitrary. In fact, here we see an explicit example where $\dim_{mc}\wt M = \dcat(M)$ holds exactly when $n=0$ and $M$ is a (flat) torus as in Theorem~\ref{dcatmacdim}.
%In fact, in the spirit of Theorem~\ref{dcatmacdim}, the equality $\dim_{mc}\wt M = \dcat(M)$ holds here if and only if $n=0$ and $M$ is a (flat) torus.
}
\end{example}

\begin{remark}
    We conclude by pointing out that the arguments of Proposition~\ref{newnew1} and Theorem~\ref{dcatmacdim} give us the improved bound $\dim_{mc}\Wi M\le 2\dcat(M)-\dim(M)$ for any closed $c$-symplectic manifold $M$ having non-negative Ricci curvature and infinite fundamental group. However, the equality $\dim_{mc}\Wi M=2\dcat(M)-\dim(M)$ here does not imply that $M$ is flat (or even aspherical). To see this, consider $M=T^{2k}\times N$ for a closed simply connected symplectic $(2n-2k)$ manifold $N$ with non-negative Ricci curvature. Then, $\dim_{mc}\wt M = 2k$ and $\dcat(M)=n+k$, but $M$ is not aspherical. The same example also shows that in the setting of Proposition~\ref{newnew1}, the equality $b_1(M)= 2\dcat(M)-\dim(M)$ does not imply that $M$ is a torus.
\end{remark}

\section*{Acknowledgement}

The authors thank the anonymous referee for valuable suggestions.

\end{document}